\documentclass[review]{elsarticle}

\usepackage{hyperref}

\journal{Applied and Computational Harmonic Analysis}

\usepackage{amsthm, amssymb, amsmath, latexsym, amsfonts, mathcomp}
\usepackage{comment}
\usepackage{xcolor}

\newcommand{\N}{\mathbb{N}}

\newcommand{\R}{\mathbb{R}}

\newtheorem{theorem}{Theorem}[section]

\newtheorem{definition}[theorem]{Definition}
\newtheorem{proposition}[theorem]{Proposition}

\DeclareMathOperator{\suppp}{supp \,}

\newcommand*{\ms}[1]{{\color{black} #1}}
\newcommand*{\cl}[1]{{\color{black} #1}}
\newcommand*{\pp}[1]{{\color{black} #1}}

\begin{document}

\begin{frontmatter}

  \title{Bendlets:\\[5pt] \large{A Second-Order Shearlet Transform with Bent Elements}}

\author{Christian Lessig\fnref{cl}}
\author{Philipp Petersen\fnref{pp}}
\author{Martin Sch{\"a}fer\fnref{ms}}
\address{Technische Universit{\"a}t Berlin}
\fntext[cl]{christian.lessig@tu-berlin.de}
\fntext[pp]{philipp.petersen@tu-berlin.de}
\fntext[ms]{schaefer@math.tu-berlin.de}




\begin{abstract}
We introduce bendlets, a shearlet-like system that is based on anisotropic scaling, translation, shearing, \emph{and bending} of a compactly supported generator.
With shearing being linear and bending quadratic in spatial coordinates, bendlets provide what we term a second-order shearlet system.
As we show in this article, the decay rates of the associated transform enable the precise characterization of location, orientation \emph{and curvature} of discontinuities in piecewise constant images.
These results yield an improvement over existing directional representation systems where curvature
only controls the constant of the decay rate of the transform. 
{We also detail the construction of shearlet systems of arbitrary order.}
A practical implementation of bendlets is provided as an extension of the ShearLab toolbox, which we use to verify our theoretical classification results
\end{abstract}


\end{frontmatter}


\section{Introduction}

Many images consist of regions separated by piecewise smooth curves.
In photographs the regions correspond to different objects and spatial occlusion between them gives rise to edges; in medical imaging curves separate bones and different kinds of soft tissue; and in geoscience the boundaries between different geological strata lead to discontinuities in measured data, to name a few examples.
The curves separating the different regions thus provide valuable information about an image's structure and content~\cite{Nitzberg1993} so that knowing their location and geometric properties facilitates many image processing, analysis and understanding tasks.

In the last years, it has been shown that directional representation systems~\cite{Candes2005a,Candes2005b,Do2005a,Kutyniok2012} are highly efficient at extracting and characterizing boundary curves in piecewise smooth images. In contrast to regular wavelets, these systems are not only able to precisely locate the boundaries but also provide information on their normal directions. 
In particular, the points and orientations in parameter space where the associated directional transforms decay slowly correspond to the positions and normal directions of the curves.
For shearlets, it has been furthermore shown that they have the ability to detect non-smooth corner points on the curves and to analyze their properties~\cite{GuoL2009detect2D,Kutyniok2016}.

In contrast to their ability to detect local orientation, existing directional representation systems cannot precisely classify curvature, despite the importance of this information for the description and analysis of boundary segments~\cite{Kutyniok2016}. 
In this paper, we address this shortcoming and introduce an extension of shearlets that allows for the precise classification also of curvature.
This is enabled by improving the adaptivity of the system.
Conventional directional systems typically accomplish adaptivity by
utilizing some form of anisotropic scaling, e.g. $\alpha$-scaling with $\alpha\in[0,1)$~\cite{GKKSAlpha16}, and controlling the location and orientation of the system elements.
To further increase the adaptivity, we introduce bending as another degree of freedom for the transform elements, which we, therefore, call bendlets.
This enables that the elements can align with the local bending of the boundary curves which is the key for the classification of curvature.
Since shearing has linear and bending quadratic dependence on spatial coordinates, bendlets can be considered as a second-order shearlet system.
Indeed, they are a special case of shearlet systems of arbitrary orders, whose construction we present in Section \ref{sec:construction}.


As Theorem~\ref{thm:mainClass}
will show, with bendlets the precise location, orientation and curvature of a discontinuity curve are characterized by the decay rates of the associated transform as a function of its parameters.
Bendlets hence provide more detailed information than is available with first-order shearlets, where the information on curvature is contained in the less accessible
constant and not the rate, cf.~\cite{Kutyniok2016}.

An interesting aspect of \cl{our} result is that not parabolic scaling but $\alpha$-scaling with $\frac{1}{3}<\alpha<\frac{1}{2}$ is required.
\cl{If we would employ parabolic scaling, the support of bendlets would not be sufficiently anisotropic to result in decay rates that separate different curvatures for different values of the bending parameter.
With $\frac{1}{3}<\alpha<\frac{1}{2}$, in contrast, the more elongated support ensures that even when location and orientation are appropriate, a bendlet overlaps the boundary curve only in a small region, resulting in a small transform, unless the bending matches the curvature.
We will discuss the necessity for non-parabolic scaling further in Section~2.3 and the formal reasons will become apparent in the proof of Theorem~3.1.}

In addition to their improved ability to characterize discontinuity curves,
bendlets also promise advantages for image approximation.
\cl{For cartoon-like images\ms{, i.e., $C^2$ regular signals with $C^2$ boundaries,} parabolically scaled directional representation systems, such as classical curvelets and shearlets, 
provide the optimal approximation rate $N^{-2}$~\cite{Donoho2001} up to a log-factor~\cite{CD04}.}
%
The reason lies in the fact that $C^2$ boundary curves can be represented locally in the form \cl{$(x,E(x))\in\R^2$} with $E(x)$ a $C^2$ function satisfying $E^\prime(0)=E(0)=0$.
A Taylor expansion thus yields the local approximation $E(x)\approx \frac{1}{2} E^{\prime\prime}(0) x^2 $, which matches the parabolic scaling law $width\approx length^2$ of the system elements.
For the same reason, $\alpha$-scaled systems are optimally adapted to resolve $C^k$ edges \ms{if $k\in(1,2]$} and $\alpha=k^{-1}$~\cite{GKKScurve2014}.
%
However, when the boundary curves are $C^k$ with $k > 2$ the optimal approximation rate of $N^{-k}$ cannot be achieved by conventional $\alpha$-scaled systems~\cite{MS2016}.
Furthermore, existing results (see e.g.\ \cite[Thm.~5.3]{MS2016}) demonstrate that this bound can not be overcome by changing the scaling behavior alone.
Turning to bendlets or even higher-order shearlets, \cl{which adapt to the boundary curve using additional degrees of freedom,} seems to provide a feasible route to move to approximation rates beyond the $N^{-2}$ barrier~\cite{Donoho2001,MS2016}; \cl{see Sec.~\ref{sec:conclusion} for a heuristic argument.}
\pp{While we deem the verification of these improved approximation rates to be beyond the scope of the present paper, we provide precise decay estimates of the bendlet coefficients. Such estimates are an essential part in all known approximation results for shearlets for piecewise smooth functions.}



\subsection{Related Work}

Since their invention in 2005, shearlet systems~\cite{GKL05} have been established in applied
and computational harmonic analysis as efficient representation systems, in
particular for image data. Their success is due to their superior approximation performance for images compared to wavelets. In fact, they share the (quasi) optimal approximation properties of curvelets~\cite{CD04}, yet they are better suited for digital implementation due to the utilization of shears instead of rotations. Moreover, in contrast to curvelets, frames of compactly supported shearlets~\cite{Kittipoom2011} are available.

The (quasi) optimal performance of shearlets with respect to $C^2$-regular
boundary curves was first shown for the classic band-limited 
construction~\cite{Guo2007}. 
Later it was also verified for systems with compact support~\cite{Kutyniok2010}.
As the heuristic Taylor-argument in the previous subsection suggests, the type of scaling plays a pivotal role for the approximation performance.
In fact, the common parabolic scaling law of shearlets and curvelets explains their similar approximation properties. 
This connection 
is made precise by the framework of $\alpha$-molecules~\cite{GKKSAlpha16}, confirming that the type of scaling is in this case a decisive property for approximation.
The influence of the scaling on approximation performance motivated the introductionof $\alpha$-scaled versions of shearlets~\cite{Kutyniok2012a} and curvelets~\cite{GKKScurve2014}. 
Their approximation
performance in the range $\alpha\in[\frac{1}{2},1)$ has been analyzed in \cite{GKKScurve2014, Kutyniok2012a, GKKSAlpha16}.

Besides approximation, shearlets also provide a powerful tool for feature analysis of functions.
It was first demonstrated in~\cite{Kutyniok2008} that, analogous to the curvelet transform~\cite{Candes2005a}, the shearlet transform provides a precise characterization of the wavefront set, which corresponds to the aforementioned boundary curves between image regions.
In particular, the wavefront set is characterized as those positions and orientations in parameter space where the transform decays slowly for increasing scales.
By now, the analysis of edges and singularities of functions utilizing a continuous shearlet transform is a well-established area of research~\cite{Kutyniok2012}.

These results have been generalized in~\cite{Kutyniok2008} for the so-called classical shearlet.
In~\cite{Grohs2011wavefront} these results were subsequently generalized to include also compactly supported shearlet generators.
It has also been analyzed to what extent more general dilation groups than those employed for the standard shearlet system can characterize the wavefront set~\cite{Fell2015}.

Recent work furthermore established that the shearlet transform provides geometric information that goes beyond the wavefront set, for which it is sufficient that the transform decays slower than any polynomial.
For a bounded domain $B \subset \R^2$ with piecewise smooth boundary $\partial B$, it was shown in~\cite{GuoL2009detect2D} and~\cite{GuoLL2009detect2D} that using a classical shearlet one can associate precise decay rates to the transform of the characteristic function $\chi_B$ of $B$ so that points on $\partial B$ and the corresponding normal direction can be detected as well as points where $\partial B$ is not smooth.
An edge detection algorithm based on shearlets was implemented in~\cite{YiLEK2009edgesnumeric}.

Numerous generalizations of the just described method have been developed, including a generalization to 3D domains in~\cite{GuoL2011detect3D}, with a particular emphasis on detecting line singularities within the 2D manifolds given by the boundaries~\cite{HouL2016curvesinsurfaces}.
In another line of research, the region $B$ is no longer required to be constant but can be smooth, see \cite{GuoL2016piecewisesmooth}.
All the results above deal with band-limited generators.
In~\cite{Kutyniok2016} the decay of a continuous shearlet transform with a compactly supported generator was analyzed.
It was demonstrated that a similar classification of edges as with band-limited generators could be achieved in both 2D and 3D, with the additional improvement that the decay rates are uniform, meaning they can be analyzed in a pre-asymptotic regime.
Furthermore, it was also established that information on the curvature of $\partial B$ could be extracted from the shearlet transform in a weak form. This is, however, very different from the results in the present paper that provide an exact description of the asymptotic behavior of a higher-order shearlet transform for different curvatures.

\subsection{Outline and Contribution}

The principal motivation for the present work is the classification of singularities in images. To this end, we introduce a novel directional multiscale system called bendlets, which is an extended shearlet system utilizing bending as an additional parameter.
%
{It can be regarded as a second-order shearlet system,
and in Section~\ref{sec:construction} 
we detail its construction as a special case of shearlets of arbitrarily high order.}

%
%
In Section~\ref{sec:ClassResults}, we investigate the ability of bendlets to classify discontinuity curves in images.
Our main result on classification, Theorem~\ref{thm:mainClass}, shows that with an anisotropic scaling of $1/3<\alpha<1/2$ bendlets have decay rates that allow to precisely identify location, direction, and curvature of a discontinuity curve.
This improves upon existing directional representation systems where curvature only controls the constant of the decay rate of the transform, cf.~\cite{Kutyniok2016},
and is thus less accessible.
Given the importance of curvature in differential geometry to describe and characterize shapes, we believe that bending is a natural and important parameter to be taken into account.


Our theoretical results are verified with numerical experiments in Section~\ref{sec:numerics}.
The implementation of bendlets as an extension to ShearLab~\cite{Kutyniok2016b} indicates that these might, in fact, provide a practical alternative to shearlets.
Towards this end, and in the spirit of reproducible computational science~\cite{Donoho2009}, we release our source code on \texttt{www.shearlab.org}.\footnote{\url{http://www3.math.tu-berlin.de/numerik/www.shearlab.org/software\#bendlets}}

\section{Construction}
\label{sec:construction}

\ms{The main objective of this section is the introduction
of higher-order shearlets as a generalization of the classical shearlet
systems.
Their construction} differs from that of classical shearlets in that $\alpha$-scaling \ms{instead of parabolic scaling} and a higher-order variant of the shearing operator are employed.

%


For $a>0$ and \ms{$\alpha\in[0,1]$} 
the \emph{$\alpha$-scaling operator} is defined by
\begin{align}\label{eqdef:alpmat}
A_{a,\alpha}: = \left(\begin{array}{l l} a & 0\\ 0 &a^\alpha \end{array}\right)
\end{align}
where the parameter $\alpha$ determines the scaling anisotropy. 
For example, the choice $\alpha=1$ leads to isotropic scaling, $\alpha=\frac{1}{2}$ to parabolic scaling,
and $\alpha=0$ to pure directional scaling.

For $l\in \N$ and $r=(r_1,\ldots,r_l)^T\in\R^l$ we further define the \emph{$l$-th order shearing operator} $S_{r}^{(l)}: \R^2 \to \R^2$ by
\begin{align}\label{eqdef:genshear}
S_{r}^{(l)}(x) &:= \left(\begin{array}{c c} 1 & \sum_{m = 1}^l r_m x_2^{m-1}\\ 0 & 1\end{array}\right) (x_1,x_2)^T .
\end{align}
For $l=1$ this yields an ordinary shearing matrix and for $l=2$ we obtain an operator implementing shearing and bending.

\subsection{The higher-order shearlet \ms{systems and transforms}}

We begin the construction of higher-order shearlet system by introducing the generating parameter set.

\begin{definition}
The \emph{$l$-th order $\alpha$-shearlet parameter set} $\mathbb{S}^{(l,\alpha)}$ is
$$
\mathbb{S}^{(l,\alpha)}: = \R^+ \times \R^l \times \R^2 .
$$
\end{definition}

In contrast to classical shearlets that can be regarded as being generated by the action of a group on a generating function~\cite{ShearletGroup, DKST2009}, no such structure exists for higher order shearlets with $l>1$.
The generalized shearing in~\eqref{eqdef:genshear} yields for $l > 1$ a composition that does not satisfy associativity.
This restriction notwithstanding, for $l=1$ and $\alpha=1/2$ the above definition yields the classical shearlet group.

\begin{definition}\label{def:rep} 
Let $\psi \in L^2(\R^2)$, $l\in\N$, \ms{$\alpha\in[0,1]$}. Then the action
\[
\pi^{(l,\alpha)}: \mathbb{S}^{(l,\alpha)} \to \mathrm{Aut}(L^2(\R^2))
\]
of the higher order shearlet parameter set $\mathbb{S}^{(l,\alpha)}$ on $L_2(\R^2)$ is given by
\begin{align}\label{eq:theRep}
    \pi^{(l,\alpha)}(a,r,t)\psi & := a^{-(1+\alpha)/2}\psi\left( A_{a,\alpha}^{-1} S_{-r}^{(l)}(\cdot - t)\right).
\end{align}
\end{definition}

With the above definitions we can now introduce $l$-th order $\alpha$-shearlet systems.

\begin{definition}
Let $\psi \in L^2(\R^2)$, $l\in\N$, \ms{$\alpha\in[0,1]$}.
Using the representation $\pi^{(l,\alpha)}$ from Definition~\ref{def:rep}, we define the \emph{\ms{continuous} $l$-th order $\alpha$-shearlet system}:
\begin{align*}
    \mathrm{SH}_{\psi}^{(l,\alpha)} := \Big\{ \pi^{(l,\alpha)}(a,r,t)\psi \, \left\vert \, (a,r,t) \in \mathbb{S}^{(l,\alpha)} \right. \Big\} .
\end{align*}
The associated \emph{continuous $l$-th order $\alpha$-shearlet transform} $\mathcal{SH}_{\psi}^{(l,\alpha)}(f)$ of a function $f\in L^2(\R^2)$ is the function on 
$\mathbb{S}^{(l,\alpha)}$ given by
\begin{align*}
    \mathcal{SH}_{\psi}^{(l,\alpha)}\big(f\big)(a,r,t) := \big\langle f \, , \, \pi^{(l,\alpha)}(a,r,t)\psi \big\rangle.
\end{align*}
\end{definition}

\ms{The systems $\mathrm{SH}_{\psi}^{(1,\alpha)}$ with $l=1$
are just conventional first-order $\alpha$-shearlet systems.
Systems of higher order, i.e., $\mathrm{SH}_{\psi}^{(l,\alpha)}$ with $l\ge2$,
extend those and 
provide larger dictionaries, 
containing functions better suited for the sparse representation of curved discontinuities (see Section~\ref{sec:conclusion}).} 


\ms{Concerning the associated transforms,
it is easy to verify that for a signal $f\in L^2(\R^2)$
the transform $\mathcal{SH}_{\psi}^{(l,\alpha)}(f)$ is a bounded measurable function on $\mathbb{S}^{(l,\alpha)}$.
Increasing the order \cl{provides} additional information about $f$, which, \cl{for example,} can be exploited for signal analysis.}

\ms{Subsequently, we are interested in the extraction of curvature information from discontinuity curves in a signal, which the first-order shearlet transforms cannot provide.
As we will see, a second-order transform is well suited for this task.
In the remainder of the paper, we therefore restrict our investigation to the case $l = 2$.
Nonetheless, we expect that the additional information which is provided by higher-order transforms with $l>2$ will also be useful for more complex analysis tasks, for example, if one is
interested in an even finer characterization of discontinuities.}

\ms{
To simplify notation in the second-order case, we subsequently write $r = (s,b)$ such that $s$ takes the role of shearing and $b$ corresponds to a bending of the functions $\pi^{(2,\alpha)}(a,r,t)\psi$. Since bending is the most characteristic property, we call the associated transform a \emph{bendlet transform} and the associated system a \emph{bendlet system}. \ms{We use the simplifying} notation
\begin{align*} 
\mathrm{BS}_{\psi}^{(\alpha)} := \mathrm{SH}_{\psi}^{(2,\alpha)}
=\Big\{ \psi_{a,s,b,t} \,\left\vert \, (a,s,b,t) \in \mathbb{S}^{(2,\alpha)} \right. \Big\},
\end{align*}
where \ms{$\psi_{a,s,b,t}:=\pi^{(2,\alpha)}(a,s,b,t)\psi$} and \ms{for $(a,s,b,t) \in \mathbb{S}^{(2,\alpha)}$}
\begin{align}\label{eq:bendlet_transform}
    \mathcal{BS}_{\psi}^{(\alpha)}\big(f\big)(a,s,b,t) \ms{:= \ms{\mathcal{SH}_{\psi}^{(2,\alpha)}\big(f\big)(a,s,b,t)}} = \big\langle f \, , \, \psi_{a,s,b,t} \big\rangle .
\end{align}
}



\subsection{\ms{Heuristic} classification of curvature}
\label{sec:classofCurv}

In \ms{this} \cl{subsection},
we \ms{start with a heuristic argument to provide some intuition} why the bendlet transform $\mathcal{BS}_{\psi}^{(\alpha)}$ enables to classify position, normal, and curvature \ms{of a boundary curve in an image.}

\begin{figure}
\centering
\includegraphics[width = \textwidth]{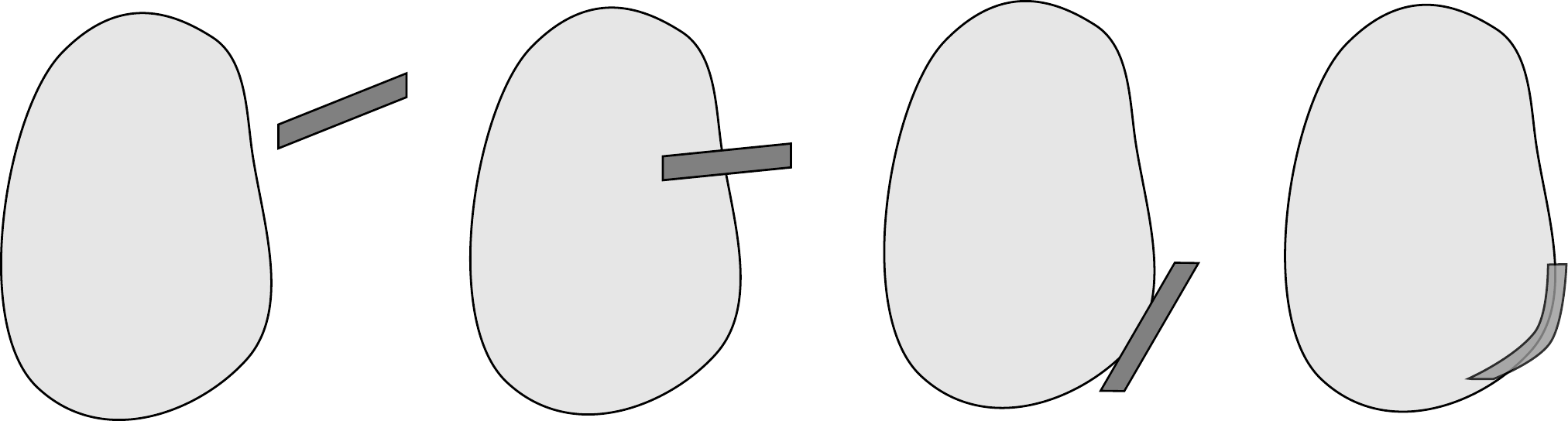}
\put(-360,90){a)}
\put(-245,90){b)}
\put(-160,90){c)}
\put(-70,90){d)}
\caption{Different scenarios that we wish to classify depending on the behavior of the shearlet transform for $a\to 0$.}
\label{fig:Cases}
\end{figure}
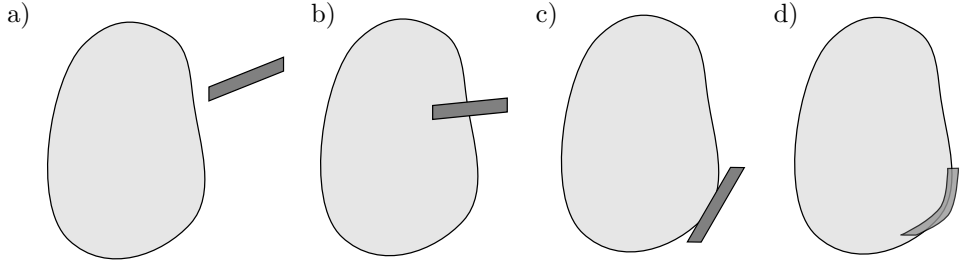

Assume we are given a domain $D\subset \R^2$ with $C^{\infty}$ boundary $\partial D$ and denote the characteristic function of $D$ by $\chi_D$.
We consider the bendlet transform
\begin{align*}
    \mathcal{BS}_\psi^{(\alpha)}\big( \chi_D \big)(a,s,b,t) = \big\langle \chi_D \, , \, \psi_{a,s,b,t} \big\rangle \ms{\quad\text{for}\quad (a,s,b,t)\in\mathbb{S}^{(2,\alpha)}}.
\end{align*}
The goal is to extract information about the boundary $\partial D$, in particular its position, normal direction and curvature, using the asymptotic behavior of
$$
\vert \mathcal{BS}_\psi^{(\alpha)}\big( \chi_D \big)(a,s,b,t) \vert \quad\text{for}\quad a \to 0,
$$
i.e.\ in the asymptotic regime of the scale parameter.

We consider four different cases, as depicted in Figure \ref{fig:Cases}:
\begin{enumerate}
\item[a)] The translation parameter $t$ is not on the boundary curve, i.e., $t \not \in \partial D$.
\item[b)] $t \in \partial D$ but $s$ does not correspond to the normal of $\partial D$ at $t$.
\item[c)] $t \in \partial D$ and $s$ corresponds to the normal of $\partial D$ at $t$, but $b$ does not correspond to the curvature of $\partial D$ at $t$.
\item[d)] $t \in \partial D$ and $s$ corresponds to the normal of $\partial D$ at $t$ and $b$ corresponds to the curvature of $\partial D$ in $t$.
\end{enumerate}

In Theorem \ref{thm:mainClass} in Section \ref{sec:ClassResults} we will demonstrate that the bendlet transform indeed provides a different asymptotic behavior for each of the cases above. Nonetheless, one can already obtain a good intuition why such a classification is possible by \ms{heuristically} analyzing a simple model case.

Let us examine the transform \ms{$\mathcal{BS}_\psi^{(\alpha)}\big( \chi_D \big)$} \ms{of the unit ball $D = B_1(0)$ in $\R^2$ centered at the origin,} in which case $\partial D = S_1: = \{x\in \R^2: \|x\| = 1\}$.
As generator of the shearlet system, we use a function $\psi \in L^2(\R^2)$ with compact support \ms{and at least} one vanishing moment, i.e., $\int \psi \, dx = 0$.
Usually, we also assume some conditions on \ms{the} smoothness of $\psi$. \ms{For }the following heuristic argument, we will omit an explicit requirement on the smoothness, \ms{however}.
\ms{It will} only play a role in the \ms{rigorous, somewhat technical,} proof of the decay rate of case b), \ms{discussed in Sec.~\ref{sec:ClassResults}}.

\ms{In case a), i.e., if $t\not \in  \partial D$,} it is clear that for $a$ sufficiently small $\suppp \psi_{a,s,b,t} \cap S_1 = \emptyset$. Hence\ms{, since $\psi$ has at least one vanishing moment, the transform $\mathcal{BS}_\psi^{(\alpha)}\big( \chi_D \big)(a,s,b,t)$} will be zero for sufficiently small $a$.

\ms{Let us proceed with} case c), where we assume $t = (1,0)$ and $s = 0$. 
%
%
%
Using \ms{the second-order} Taylor approximation
\begin{align}\label{eq:secordTay}
x_1\approx 1- \frac{1}{2}x_2^2
\end{align}
of the boundary curve of $B_1(0)$ in $(1,0)$,
we can approximate $\chi_{B_1(0)}$ in a neighborhood of $(1,0)$ by
$$
\chi_{\tilde{D}}, \text{ where } \tilde{D}: = \{ x\in \R^2: x_1 - 1 \leq -\frac{1}{2}x_2^2 \}.
$$
Furthermore, there exists a constant $C>0$, independent of $a$, such that
$$
\suppp \psi_{a,0,0,(1,0)}\subseteq [1-Ca,1+Ca]\times [-Ca^\alpha,Ca^\alpha]
$$
since the shearlets $\psi_{a,0,0,(1,0)}$ are $\alpha$-scaled and stem from a compactly supported generator $\psi$.
Let us estimate the size of the overlap $|\suppp \psi_{a,0,0,(1,0)} \cap B_1(0)|$. 
We easily observe that 
$$
|\suppp \psi_{a,0,0,(1,0)} \cap \tilde{D} | \subseteq \{ x\in \R^2: 1- Ca \le x_1 \le 1-\frac{1}{2}x_2^2 \}.  
$$
This yields the estimate
$$
|\suppp \psi_{a,0,0,(1,0)} \cap B_1(0)| \approx |\suppp \psi_{a,0,0,(1,0)} \cap \tilde{D} |  \lesssim a a^{1/2} = a^{\frac{3}{2}} \,,
$$
and consequently
$$
|\langle \psi_{a,0,0,(1,0)}, \chi_{B_1(0)} \rangle| \lesssim a^{-\frac{1+\alpha}{2}}a^{\frac{3}{2}} = a^{\frac{2-\alpha}{2}}.
$$
Next, we examine case d). Let us again assume $t = (1,0),\, s = 0$, but now  with the \ms{parameter $b =- \frac{1}{2}$ matching the second-order Taylor
coefficient in~\eqref{eq:secordTay}}. 
As before, we use the local approximation $B_1(0)\approx \tilde{D}$ near $(0,1)$ and
compute
\begin{align*}
\langle \psi_{a,0,-\frac{1}{2},(1,0)}, \chi_{B_1(0)} \rangle \approx \ &\langle
\psi_{a,0,-\frac{1}{2},(1,0)}, \chi_{\tilde{D}} \rangle\\
= \ &a^{-\frac{1+\alpha}{2}}\int_{ x_1 \leq 1-\frac{1}{2} x_2^2} \psi\left( A_{a,\alpha}^{-1}
S_{(0,\frac{1}{2})}^{(2)}(\cdot - (1,0))\right) \, dx .
\end{align*}
We simplify by applying the shift $x_1 \mapsto x_1 +1$ which yields
\begin{align*}
\langle \psi_{a,0,-\frac{1}{2},(1,0)}, \chi_{B_1(0)} \rangle \approx \
&a^{-\frac{1+\alpha}{2}}\int_{ x_1 \leq - \frac{1}{2} x_2^2} \psi\left( A_{a,\alpha}^{-1}
S_{(0,\frac{1}{2})}^{(2)}(x)\right) \, dx .
\end{align*}
We then apply the bending transformation $x \mapsto S_{-(0,\frac{1}{2})}^{(2)}(x)$ and obtain
\begin{align*}
&a^{-\frac{1+\alpha}{2}}\int_{ x_1 \leq -\frac{1}{2} x_2^2}\psi\left( A_{a,\alpha}^{-1}
S_{(0,\frac{1}{2})}^{(2)}(x)\right) = a^{-\frac{1+\alpha}{2}}\int_{ x_1 \leq 0} \psi\left(
A_{a,\alpha}^{-1} x\right)\, dx.
\end{align*}
We can make another transformation, using the scaling $x \mapsto A_{a,\alpha}x$, to
obtain
\begin{align*}
\langle \psi_{a,0,-\frac{1}{2},(1,0)}, \chi_{B_1(0)} \rangle \approx
a^{\frac{1+\alpha}{2}}\int_{ x_1 \leq 0} \psi(x) \, dx.
\end{align*}

If we make the assumptions $\int_{ x_1 \leq 0} \psi(x) \, dx \neq 0$ and $\alpha
<1/2$, \ms{which ensures $2-\alpha>1+\alpha$}, we see that \ms{for $a\to0$ the decay of $\mathcal{BS}_\psi^{(\alpha)}\big( \chi_D \big)(a,s,b,t)$ in case d) is slower compared to the cases a) and c).
This observation can be explained by the fact that the overlap of the shearlets with the boundary curve shrinks more quickly in the other cases.
Moreover, the heuristic for case d) provides us with a lower bound.
Only this bound allows} us to distinguish this case \ms{strictly} from the others via the decay of the transform \ms{$\mathcal{BS}_\psi^{(\alpha)}\big( \chi_D \big)$}.
%

We will later see that a
similar treatment of case b) is possible as well and one also obtains faster decay in
that case than in case d). Consequently, \ms{the first-order and second-order Taylor coefficients} 
at a point $t$ on the boundary can be found by examining the transform \ms{$\mathcal{BS}_\psi^{(\alpha)}\big( \chi_D \big)(a,s,b,t)$} for different values $(s,b)$ and picking the one with the slowest decay
\ms{for $a\to0$}. The values $s$ and $b$
then describe the normal direction and curvature at position $t$ \ms{via the general formula given in \eqref{eq:curvature} below.}

\subsection{Cone-adapted transform}

The classification procedure outlined in the previous subsection will not work if the unit normal is $(0,1)$ or $(0,-1)$ at some point along the boundary $\partial D$. \cl{
\ms{The reason is that} orientation is changed by shearing along the \ms{first axis and hence no bendlet can have an orientation aligned with it}.
Furthermore, normals close to $(0,1)$ or $(0,-1)$ require an extreme shearing that renders the transform ill-suited for numerical computations.
Following previous work on shearlets~\cite{Guo2006, Kutyniok2008, GuoL2009detect2D}, we hence introduce a cone-adapted version of the bendlet transform $\mathcal{BS}_{\psi}^{(\alpha)}$ that avoids these problems by working over two cones obtained by dividing $\R^2$ along $y=x$ and $y=-x$: a horizontal cone that covers orientations around $(1,0)$ or $(-1,0)$ and a vertical cone for orientations around $(0,1)$ or \ms{$(0,-1)$}.}

\begin{definition}\label{def:ConeAdaptedTransform}
Let $\psi\in L^2(\R^2)$ and \ms{$\alpha \in [0,1]$}. We define $\tilde{\psi}(x_1,x_2): = \psi(x_2,x_1)$ and \ms{for $(a,s,b,t)\in\mathbb{S}^{(2,\alpha)}$} 
\begin{align*}
\tilde{\pi}^{(2,\alpha)}(a,s,b,t) \tilde{\psi}: = a^{-(1+\alpha)/2}\tilde{\psi}\left( \tilde{A}_{a,\alpha}^{-1} \tilde{S}_{-(s,b)}^{(2)}(\cdot - t)\right),
\end{align*}
where
$$
\tilde{A}_{a,\alpha} := \left( \begin{array}{c c}
a^{\alpha} & 0 \\
0 & a
\end{array} \right), \quad \tilde{S}_{(s,b)}^{(2)} := \left( S_{(s,b)}^{(2)} \right)^T.
$$
\ms{For $(a,s,b,t)\in\mathbb{S}^{(2,\alpha)}$ we then introduce
\begin{align*}
\psi_{a,s,b,t,1} :=\pi^{(2,\alpha)}(a,s,b,t)\psi, \quad \psi_{a,s,b,t,-1} :=\tilde{\pi}^{(2,\alpha)}(a,s,b,t)\tilde{\psi},
\end{align*}
and define the \emph{cone-adapted bendlet system}} by
$$
\Big\{\psi_{a,s,b,t,\iota} \,\left\vert \, a \in (0,1), s\in [-1,1], b\in \R, t\in \R^2, \iota \in \{-1,1\} \right. \Big\}.
$$
Moreover, we define the \emph{cone-adapted bendlet transform} by
\ms{
\begin{align*}
\mathcal{BS}^{(\alpha)}_\psi \big(f\big)(a,s,b,t,\iota) := \big\langle f \, , \, \psi_{a,s,b,t,\iota} \big\rangle,
\end{align*}
where $a\in(0,1)$, $s\in [-1,1]$, $b\in\R$, $t\in \R^2$, and $\iota\in\{-1,1\}$.}
\end{definition}

\ms{
In the above definition, we only consider the second-order case, i.e., bendlets, but analogous definitions can be made for arbitrary orders.
Note, that for each cone we only need shears in the range $s\in [-1,1]$. Further, we can conveniently restrict to scales $a\in(0,1)$ since only the high scales $a\to 0$} \cl{are of relevance in the remainder of the paper.}

To simplify the exposition in the following arguments, we end this section with the following very practical definition.
\begin{definition}
  \label{def:boundary_local_second_order}
Let $D \subset \R^2$ such that $\partial D \in C^{\infty}$. Then for every point $(p_1,p_2) = p\in \partial D$ we have that $\partial D$ is locally the graph of at least one of the functions
\begin{align*}
g_{1}(x_2) &= p_1 + s' (x_2-p_2) + b' (x_2 - p_2)^2 + H(x_2-p_2) \text{ or }\\
g_{-1}(x_1) &= p_2 + s' (x_1-p_1) + b' (x_1-p_1)^2 + H(x_1 - p_1),
\end{align*}
where $s' \in [-1,1]$, $b' \in \R$ and $|H(z)/z^3| \leq c_H$ for $z\to 0$.
If $\partial D$ is locally given by $g_1$ in a neighborhood of $p$, then we call $p$ \emph{a point of type $(s', b', 1)$}. If it is given by $g_{-1}$, then it is called a \emph{point of type \ms{$(s', b', -1)$}}.
\end{definition}

\cl{The information provided by $(s', b', \cdot )$ is sufficient to determine normal and curvature at $p \in \partial D$.
A normal vector at $p \in \partial D$ is given by $(1,-s')$ or $(-s',1)$, depending on whether $g_{1}$ or $g_{-1}$ is used.
Further, matching the second order Taylor expansion of the boundary to a circle \ms{of radius $r>0$}, whose curvature is $1/r$, one obtains
\begin{align}\label{eq:curvature}
  K = \ms{\frac{2 \vert b' \vert}{(1+(s')^2)^{3/2}}}
\end{align}
for the curvature at $p\in\partial D$.}


\section{The Classification Result}\label{sec:ClassResults}

The following theorem describes the classification of different points by the cone-adapted transform of Definition \ref{def:ConeAdaptedTransform}.

\begin{theorem}\label{thm:mainClass}
Let $0<\alpha\leq1/2$, $D \subset \R^2$ and $\partial D\in C^{\infty}$, let $\psi = \psi^1\otimes \phi^1$, where $\psi^1, \phi^1 \in L^2(\R)$ have compact support and $\psi^1$ is bounded and has $M\in \N$ vanishing moments and $\phi^1 \in C^L$ with $L>M$. Then the following holds.
\begin{enumerate}
\item If $p\not \in \partial D$ then for all $a\in (0,1)$, $s\in [-1,1]$ and $b\in \R$, $\iota \in \{-1,1\}$
$$
|\mathcal{BS}^{(\alpha)}_\psi(\chi_D)(a,s,b,p,\iota)| \lesssim a^{N},
$$
for all $N\in \N$.
\item If $p \in \partial D$ and $p$ is a point of type $(s',b',\iota')$ then
\begin{enumerate}
\item If $(s,\iota) \neq (s',\iota ')$ then
\begin{align*} 
|\mathcal{BS}^{(\alpha)}_\psi(\chi_D)(a,s,b,p,\iota)| \lesssim a^{(1-\alpha)(M+1) + \frac{1+\alpha}{2}}.
\end{align*}
\item If $(s,\iota) = (s',\iota ')$ but $b' \neq b$ then
$$
|\mathcal{BS}^{(\alpha)}_\psi(\chi_D)(a,s,b,p,\iota)| \lesssim a^{\frac{2-\alpha}{2}}.
$$
If, furthermore, $\phi^1(0)$, $\int \limits_{x_1 \leq -x_2^2} \psi^1(x_1) dx$, and $\int \limits_{x_1 \geq x_2^2} \psi^1(x_1) dx$ are all non-zero, then we also have
\begin{align} \label{eq:thePositiveLimit}
\lim_{a \to 0} a^{\frac{\alpha-2}{2}} |\mathcal{BS}^{(\alpha)}_\psi(\chi_D)(a,s,b,p,\iota)|>0.
\end{align}
\item If $(s,b,\iota) = (s',b',\iota')$ then
$$
|\mathcal{BS}^{(\alpha)}_\psi(\chi_D)(a,s,b,p,\iota)| \lesssim a^{\frac{1+\alpha}{2}}.
$$
Furthermore, if $\alpha > 1/3$ we also have the lower bound
$$
|\mathcal{BS}^{(\alpha)}_\psi(\chi_D)(a,s,b,p,\iota)| \gtrsim a^{\frac{1+\alpha}{2}} |\int_{x_1 \geq 0} \psi(x) dx|.
$$
\end{enumerate}
\end{enumerate}
\end{theorem}

\noindent
For example, for $\alpha = 1/3$ and $M=8$ one obtains the following decay rates
\begin{center}
\begin{tabular}{l|l|c}
	1. & $p \not\in \partial D$ & $a^N$
	\\
	\hline
	2.(a) & $p \in \partial D$, $(s,\iota) \neq (s',\iota ')$ & $a^{20/3}$
	\\
	\hline
	2.(b) & $p \in \partial D$, $(s,\iota) = (s',\iota ')$, $b \neq b'$ & $a^{5/6}$
	\\
	\hline
	2.(c) & $p \in \partial D$, $(s,\iota) = (s',\iota ')$, $b = b'$ & $a^{2/3}$
\end{tabular}
\end{center}
Hence, one has a clear separation of the decay rates for the different situations,
\ms{which allows to locate the boundary $\partial D$ and, in addition,
to extract the type $(s',b',\iota')$, in the sense of Definition~\ref{def:boundary_local_second_order}, of points $x\in\partial D$.}
\cl{This, in turn, allows to determine the normal direction as $(1,-s')$ or $(-s',1)$, depending on $\iota'$, and the curvature using Eq.~\ref{eq:curvature}.}

\ms{After this illustration, let us turn to the proof of} Theorem~\ref{thm:mainClass}.

\begin{proof}

\textbf{Part 1.}

This is trivial, if $\psi$ has one vanishing moment and compact support, since then
$$
\mathcal{BS}^{(\alpha)}_\psi(\chi_D)(a,s,b,p,\iota) = 0 \text{ for } a \text{ small enough. }
$$

\textbf{Part 2. (a)}

Let us first assume $\iota = \iota' = 1$ and $s \neq s'$. If $\iota = \iota' = -1$ the following computations can be made with $x_1$ and $x_2$ interchanged and we will make a comment at the appropriate position in the proof where some change needs to be made if $\iota \neq \iota'$.

We further simplify the setup based on the equality
\begin{align}\label{eq:trafoFormula}
\langle \chi_D, \psi_{a,s,b,t,1}\rangle = \langle \chi_{\tilde{D}}, \psi_{a,0,0,(0,0),1}\rangle,
\end{align}
where $\tilde{D}\subset \R^2$ is an appropriately modified domain, $0\in \partial \tilde{D}$ and $0$ is a point of type $(s'',b'',1)$ with $s''=s'-s\neq 0$.

Hence we can make the additional simplifying assumption that $s' \neq 0$ and $s,b =0$. Moving on with the simplified setup we have that in a neighborhood of $(0,0)$ the curve $\partial \tilde{D}$ is given as the graph of $g_1$, where $g_1$ is invertible on a neighborhood of $0$ since $s'\neq 0$ and
\begin{align*} 
\langle \chi_{\tilde D}, \psi_{a,0,0,(0,0),1} \rangle  = \int_{x_2 \leq g_1^{-1}(x_1)} \psi_{a,0,0,(0,0),1}(x) dx.
\end{align*}
If we had assumed $\iota \neq \iota'$ but $\iota = 1$ then we would replace $g_1^{-1}$ above by $g_{-1}$.
Next we can apply the transformation \ms{$x_2 \mapsto x_2 + g^{-1}_1(x_1)$} to obtain
\begin{align*}
\langle\chi_{\tilde{D}}, \psi_{a,0,0,(0,0),1}\rangle = \int_{x_2 \leq 0} a^{-\frac{1+\alpha}{2}}\psi^1(a^{-1}x_1)\phi^1(a^{-\alpha}(x_2+g^{-1}_1(x_1))) dx.
\end{align*}
Using the compact support of $\psi^1$ and $\phi^1$ we can find $c>0$ such that
\begin{align}\label{eq:WeNeedToUseTHeVMNow}
\langle\chi_{\tilde{D}}, &\psi_{a,0,0,(0,0),1}\rangle \nonumber\\
& \quad\;\; = \int \limits_{\substack{x_2 \leq 0\\ x \in A_{a,\alpha}([-c,c]^2)}} a^{-\frac{1+\alpha}{2}}\psi^1(a^{-1}x_1)\phi^1(a^{-\alpha}(x_2 + g^{-1}_1(x_1))) dx.
\end{align}
Let $x_2 \in [-a^\alpha c, a^\alpha c]$ be fixed. Computing Taylor approximations of $\phi^1\in C^L$ and $g^{-1}_1$ we can write for $x_1 \in [-a c,a c]$
\begin{align}\label{eq:thePolynomialExpression}
\phi^1(a^{-\alpha}(x_2-g^{-1}_1(x_1))) = \sum_{0\leq n\leq M}c_n(x_2)(a^{-\alpha} x_1)^n + O(a^{(1-\alpha)(M+1)}),
\end{align}
with $|c_n(x_2)|$ bounded from above independently of $x_2$.
Invoking the vanishing moments property of $\psi^1$ (note $L\ge M+1$) we obtain that the lower order monomials of \eqref{eq:thePolynomialExpression} get canceled and \eqref{eq:WeNeedToUseTHeVMNow} can be estimated like
\begin{align*}
|\langle\chi_{\tilde{D}}, \psi_{a,0,0,(0,0),1}\rangle| \lesssim \int \limits_{x\in A_{a,\alpha}([-c,c]^2) } a^{-\frac{1+\alpha}{2}} a^{(1-\alpha)(M+1)}dx \lesssim a^{(1-\alpha)(M+1) + \frac{1+\alpha}{2}}.
\end{align*}
This finishes the Part 2 (a) of the proof.

\textbf{Part 2. (b)}
By the argument of \eqref{eq:trafoFormula} we can assume $\iota = \iota'$, $s = s'=0$, $t = t' = (0,0)$ and $b' \neq 0, b = 0$.
Now we have
\begin{align*}
\langle \chi_{\tilde D}, \psi_{a,0,0,(0,0),1}\rangle  = \int_{x_1 \leq b' x_2^2 + H(x_2)} \psi_{a,0,0,(0,0),1}(x) dx.
\end{align*}
At this point we may assume that $b'<0$ and we will describe at the end of the proof what has to be adapted if $b'>0$. Since $\psi$ has compact support we can assume that there exists some $c>0$ such that $\suppp \psi_{a,0,0,(0,0), 1} \subset A_{a,\alpha}([-c,c]^2)$. Thus, we can compute the following estimate
\begin{align*}
&\int_{x_1 \leq b' x_2^2 + H(x_2)} \psi_{a,0,0,(0,0),1}(x) dx \nonumber\\
= &\int \limits_{\substack{x_1 \leq b' x_2^2 + H(x_2)\\ |x_1|\leq ac, |x_2|\leq a^\alpha c}} \psi_{a,0,0,(0,0),1}(x) dx \nonumber \\
=&\int \limits_{\substack{x_1 \leq b' x_2^2\\ |x_1|\leq ac, |x_2|\leq a^\alpha c}} \psi_{a,0,0,(0,0),1}(x) dx \nonumber - \int \limits_{\substack{ b' x_2^2 \leq x_1 \leq b' x_2^2 + H(x_2) \\ |x_1|\leq ac, |x_2|\leq a^\alpha c}} \psi_{a,0,0,(0,0),1}(x) dx \\ & \quad + \int \limits_{\substack{ b' x_2^2 + H(x_2) \leq x_1 \leq b' x_2^2  \\ |x_1|\leq ac, |x_2|\leq a^\alpha c}} \psi_{a,0,0,(0,0),1}(x) dx\nonumber \\
=: &\:\mathrm{I} - \mathrm{II} + \mathrm{III}.
\end{align*}
We use the property that $|H(x_2)| \leq c_H |x_2| x_2^2$ for some $c_H>0$ and hence for $|x_2| < -b'/(2c_H)$ we have that $|H(x_2)| \leq (-b'/2)x_2^2$. Since $x_1 \leq (b'/2) x_2^2$, $b'<0$ and $|x_1|\leq ac$ implies that $|x_2| \leq \sqrt{- 2ac / b'}$, we can estimate
\begin{align*}
 \mathrm{II} = &\int \limits_{\substack{ b' x_2^2 \leq x_1 \leq b' x_2^2 + H(x_2) \\ |x_1|\leq ac, |x_2|\leq a^\alpha c}} \psi_{a,0,0,(0,0),1}(x) dx\\
 = &\int \limits_{\substack{ b' x_2^2 \leq x_1 \leq b' x_2^2 + H(x_2) \\ |x_1|\leq ac, |x_2|\leq \sqrt{- 2ac / b'}}} \psi_{a,0,0,(0,0),1}(x) dx\\
 \lesssim & \int \limits_{\substack{ b' x_2^2 \leq x_1 \leq b' x_2^2 + H(x_2) \\ |x_1|\leq ac, |x_2|\leq \sqrt{- 2ac / b'}}} a^{-\frac{1+\alpha}{2}} dx \lesssim \ms{a^{-\frac{1+\alpha}{2}} a a^{\frac{1}{2}} = a^{\frac{2-\alpha}{2}},}
\end{align*}
where the second to last estimate follows due to $|H(x_2)| \leq c_H |x_2| x_2^2$. The estimate of $\mathrm{III}$ follows analogously.

We proceed by estimating $\mathrm{I}$. We have that
\begin{align*}
\ms{\mathrm{I} =  \int \limits_{\substack{x_1 \leq b' x_2^2\\ |x_1|\leq ac, |x_2|\leq \sqrt{- ac / b'}}} \psi_{a,0,0,(0,0),1}(x) dx, }
\end{align*}
since, as already previously observed, \ms{$x_1 \leq b' x_2^2$}, $b'<0$ and $|x_1|\leq ac$ implies that $|x_2| \leq \sqrt{- ac / b'}$. As a next step, we apply a transformation of
$(x_1,x_2) \mapsto (a x_1, \sqrt{-a/b'}x_2)$. This yields
\begin{align*}
\mathrm{I} = a^{\frac{2-\alpha}{2}} \ms{\sqrt{-1/b'}} \int \limits_{\substack{x_1 \leq -x_2^2\\ |x_1| \leq c, |x_2| \leq \sqrt{c} }} \psi^1(x_1) \phi^1(a^{\frac{1}{2}-\alpha} \ms{\sqrt{-1/b'}}x_2)dx.
\end{align*}
Since $\alpha< \frac{1}{2}$, we observe that, due to the continuity of \ms{$\phi^1$}, we have $\phi^1(a^{\frac{1}{2}-\alpha} \ms{\sqrt{-1/b'}} x_2) \to \phi^1(0)$ pointwise \ms{for all $(x_1,x_2)\in\R^2$}. Moreover, since $\psi^1$ is bounded, we have by dominated convergence, that
\begin{align*}
 \int \limits_{\substack{x_1 \leq -x_2^2\\ |x_1| \leq c, |x_2| \leq \sqrt{c} }} \psi^1(x_1) \phi^1(a^{\frac{1}{2}-\alpha}\sqrt{-b'} x_2)dx \to \phi^1(0)\int \limits_{\substack{x_1 \leq -x_2^2 \\ |x_1| \leq c, |x_2| \leq \sqrt{c} }} \psi^1(x_1) dx.
\end{align*}
\ms{Since furthermore $\suppp \psi^1 \subset [-c,c]$, we conclude}
\begin{align*}
a^{\frac{\alpha-2}{2}}\ms{ \mathrm{I} \sqrt{-b'}} \to \phi^1(0)\int \limits_{x_1 \leq -x_2^2} \psi^1(x_1) dx.
\end{align*}
By the compact support of $\psi^1$ we have that $\int \limits_{x_1 \leq -x_2^2} \psi^1(x_1) dx <\infty$ which yields the upper bound of the statement of \ms{2 (b)}. Additionally, we obtain \eqref{eq:thePositiveLimit} if $\phi^1(0)\int_{x_1 \leq -x_2^2} \psi^1(x_1) dx \neq 0$.

In the event that $b'>0$ one has by using that $\int\psi_{a,0,0,(0,0),1}dx = 0$
\begin{align*}
\int_{x_1 \leq b' x_2^2 + H(x_2)} \psi_{a,0,0,(0,0),1}(x) dx = \int_{x_1 \geq b' x_2^2 + H(x_2)} \psi_{a,0,0,(0,0),1}(x) dx.
\end{align*}
\ms{From this point on,} the argument is almost identical. The condition for positivity of \eqref{eq:thePositiveLimit}\ms{, however,} becomes $\phi^1(0)\int_{x_1 \geq x_2^2} \psi^1(x_1) dx \neq 0$.

\textbf{Part 2. (c)}
The upper bound follows trivially from the boundedness and compact support of $\psi$. Concerning the lower bound, we assume using the transformation of \eqref{eq:trafoFormula} that
$\iota = \iota'$, $s = s'=0$, $t = t' = (0,0)$ and $b'= b = 0$. The substitution $x\mapsto A_{a, \alpha} x$ yields
\begin{align*}
\int_{x_1 \leq H(x_2)} \psi_{a,0,0,(0,0),1}(x) dx &= a^{\frac{1+\alpha}{2}}\int_{a x_1 \leq H(a^{\alpha} x_2)} \psi_{1,0,0,(0,0),1}(x) dx
\\[7pt]
&= a^{\frac{1+\alpha}{2}}\int_{ x_1 \leq a^{-1} H(a^{\alpha} x_2)} \psi_{1,0,0,(0,0),1}(x) dx.
\end{align*}
We notice that since $|H(z)z^{-3}| \leq c_H$ for $z \to 0$ we have that
$a^{-1} |H(a^{\alpha} x_2)| \le c_H a^{-1}a^{3\alpha}|x_2|^3 \lesssim a^{3\alpha - 1} $. We obtain that
\begin{align*}
& a^{\frac{1+\alpha}{2}}\int_{ x_1 \leq a^{-1} H(a^{\alpha} x_2)} \psi_{1,0,0,(0,0),1}(x) dx\\[7pt]
= \ & a^{\frac{1+\alpha}{2}}\int_{ x_1 \leq 0} \psi_{1,0,0,(0,0),1}(x) dx + a^{\frac{1+\alpha}{2}}\int_{ 0\leq x_1 \leq a^{-1} H(a^{\alpha} x_2)}\psi_{1,0,0,(0,0),1}(x) dx \\[7pt]
&- a^{\frac{1+\alpha}{2}}\int_{ a^{-1} H(a^{\alpha} x_2)\leq x_1 \leq 0}\psi_{1,0,0,(0,0),1}(x) dx\\[7pt]
= \ & a^{\frac{1+\alpha}{2}}\int_{ x_1 \geq 0} \psi_{1,0,0,(0,0),1}(x) dx + O(a^{\frac{1+\alpha}{2}} a^{3\alpha - 1} ).
\end{align*}
If $\alpha >1/3$ this yields the result.
\end{proof}

\section{Numerical Experiments}
\label{sec:numerics}

\begin{figure}
	\begin{center}
	\includegraphics[width=0.42\textwidth]{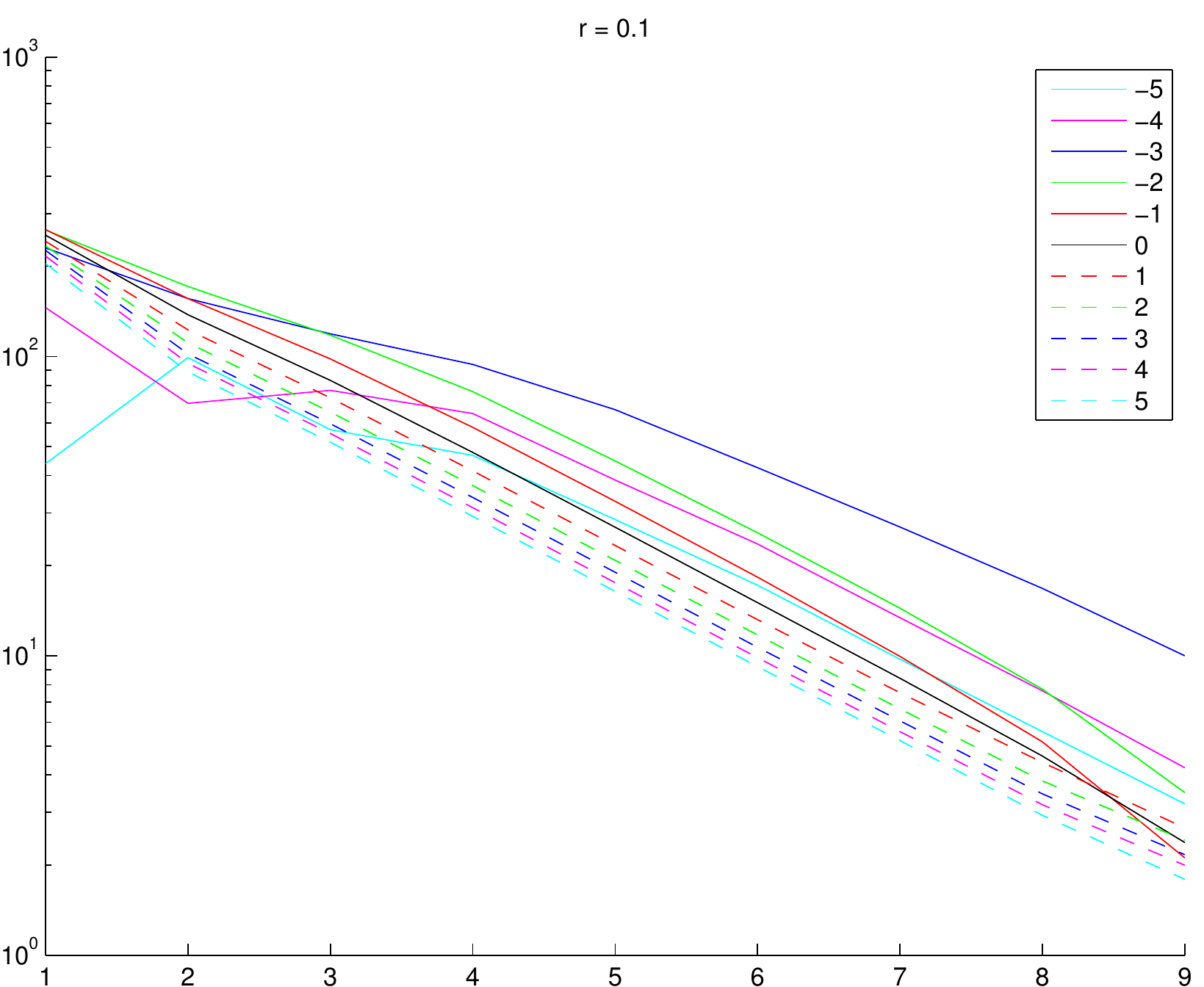}
	\includegraphics[width=0.42\textwidth]{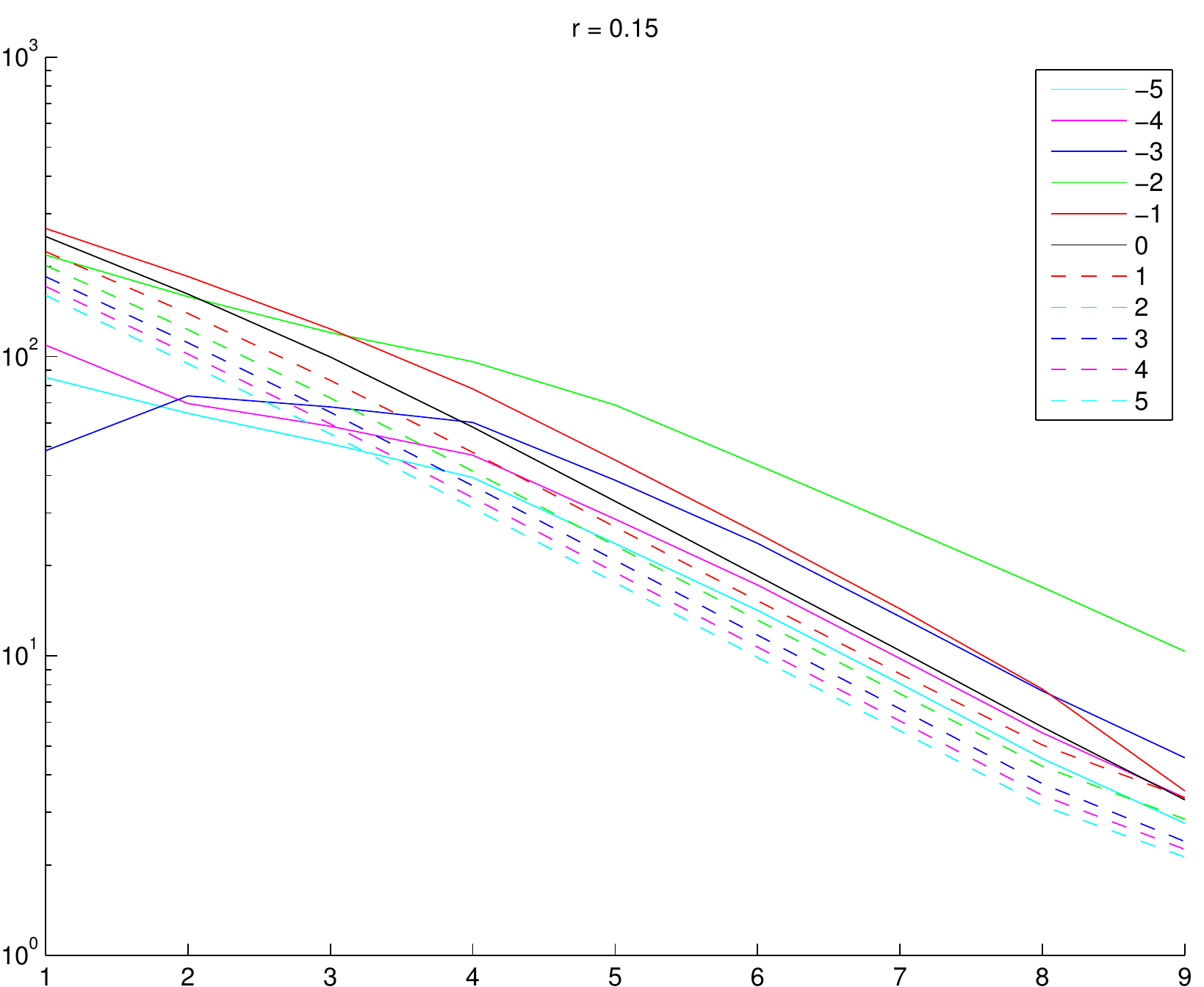}
	\\
	\includegraphics[width=0.42\textwidth]{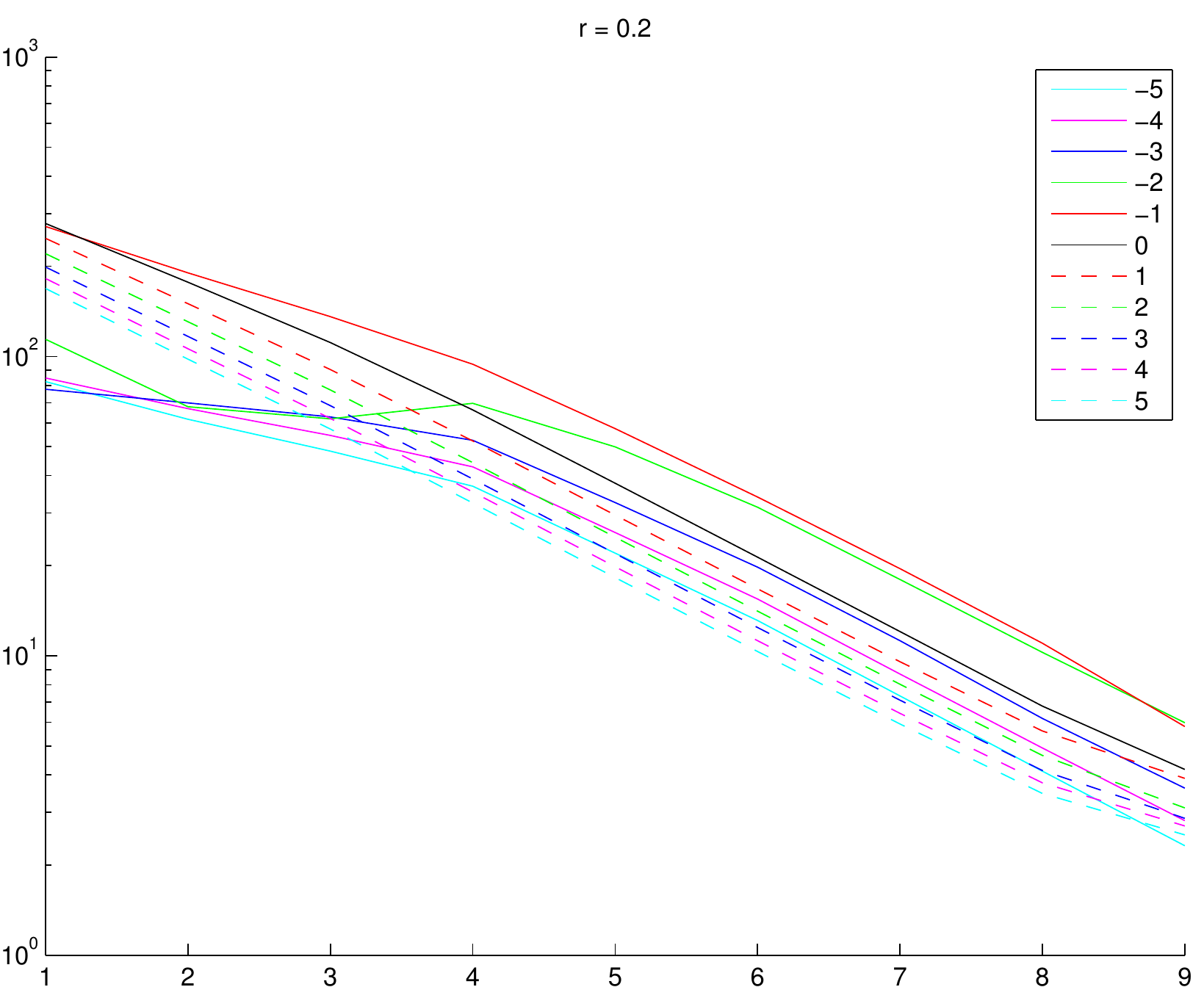}
	\includegraphics[width=0.42\textwidth]{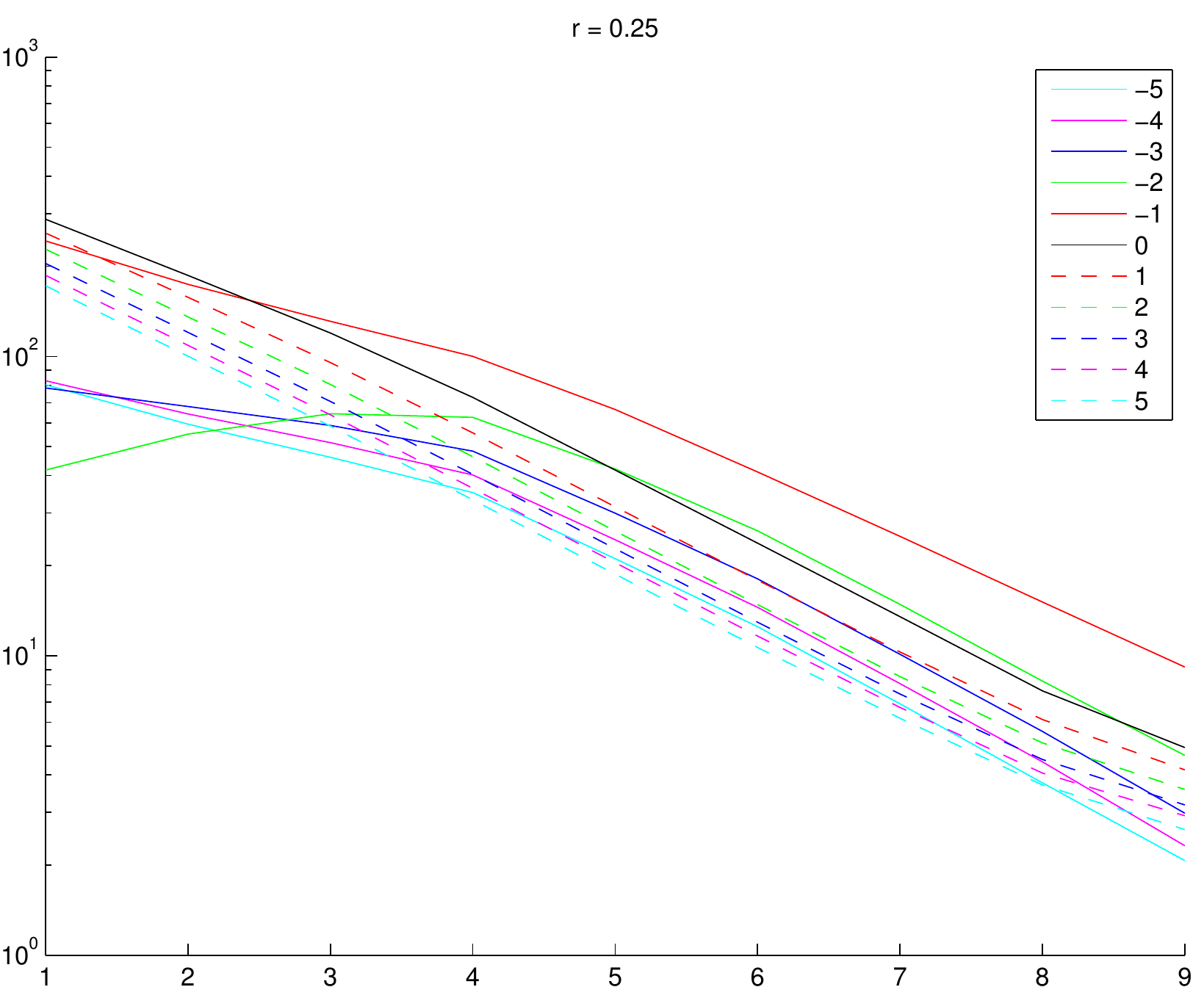}
	\\
	\includegraphics[width=0.42\textwidth]{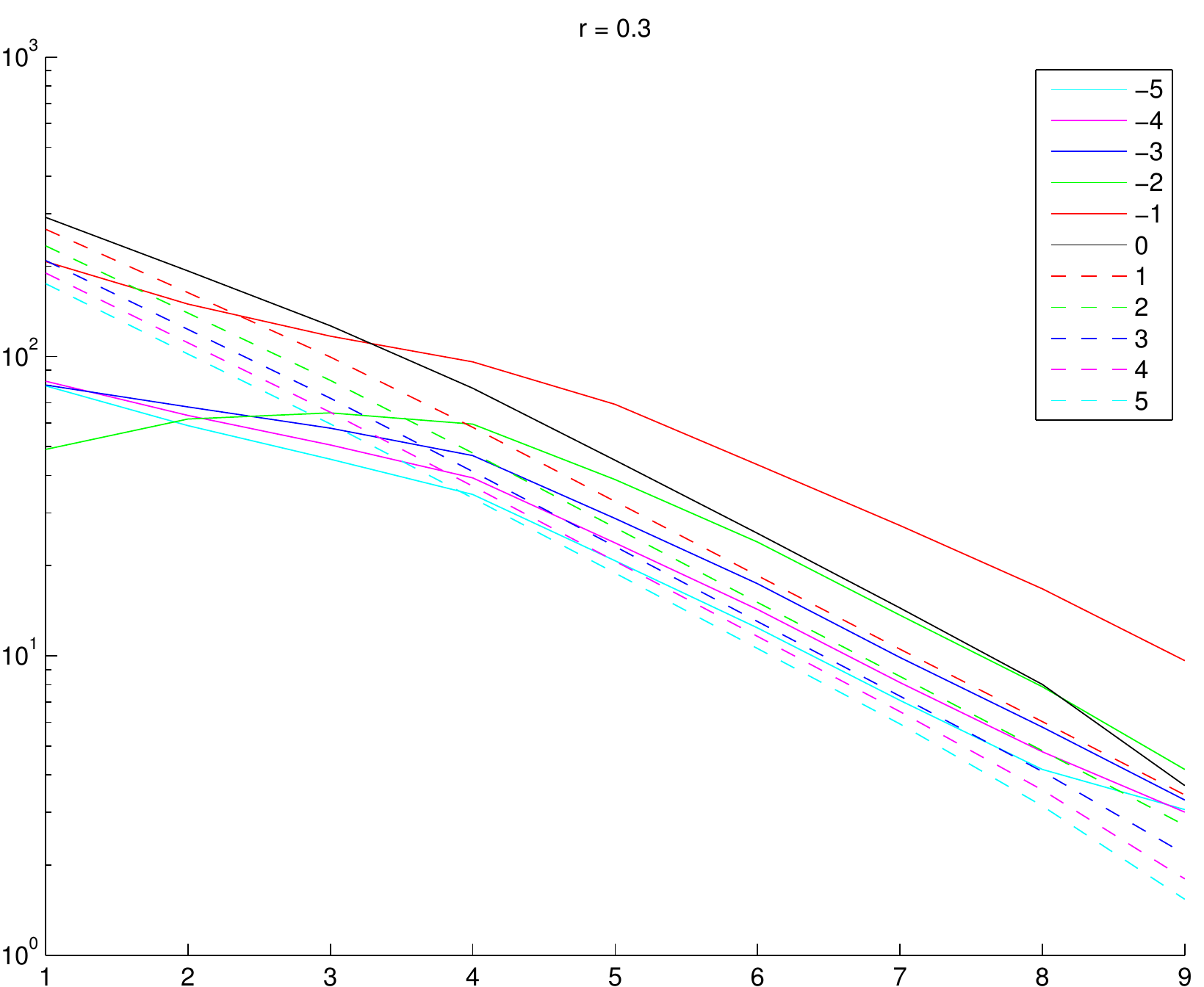}
	\includegraphics[width=0.42\textwidth]{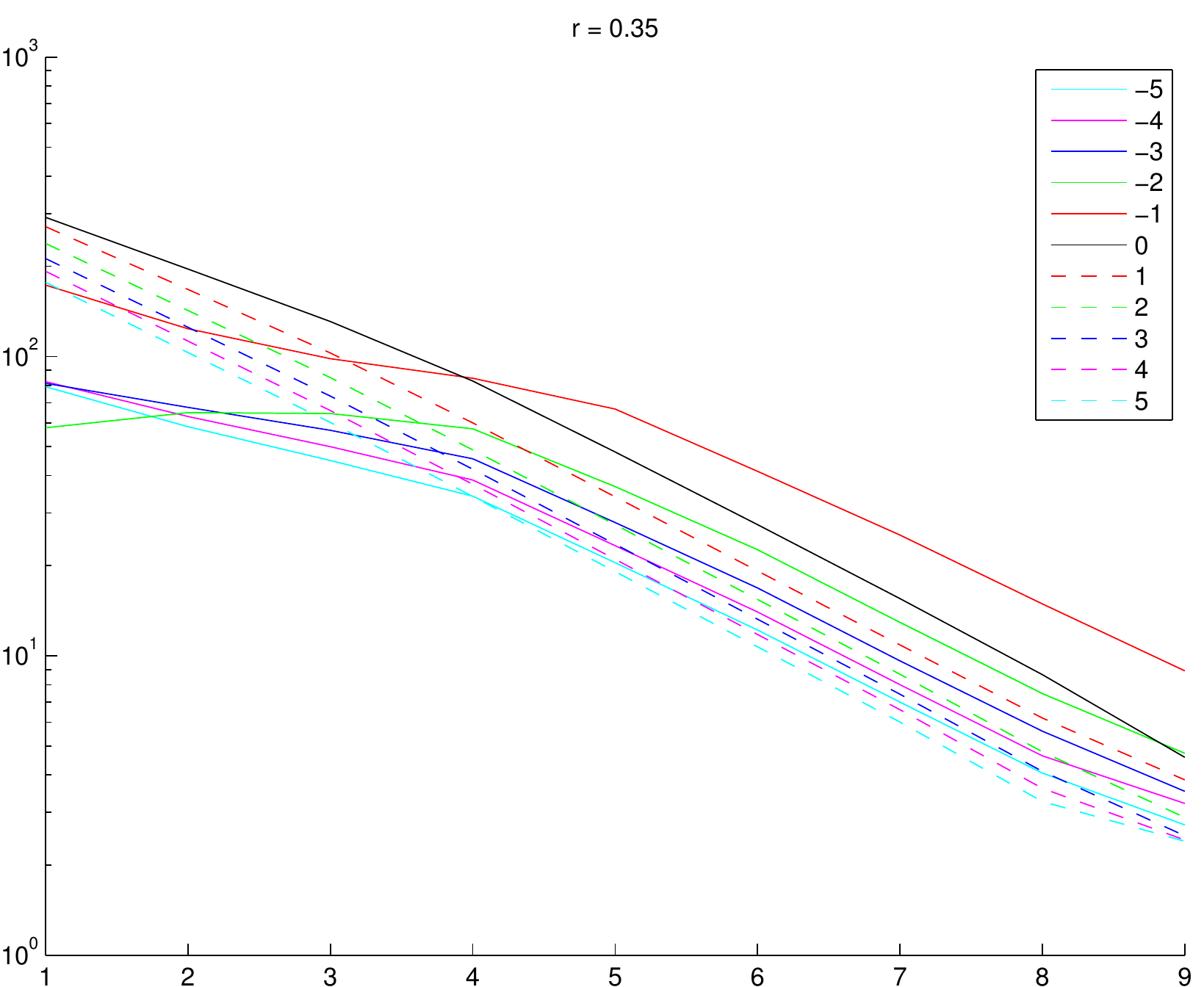}
	\\
	\includegraphics[width=0.42\textwidth]{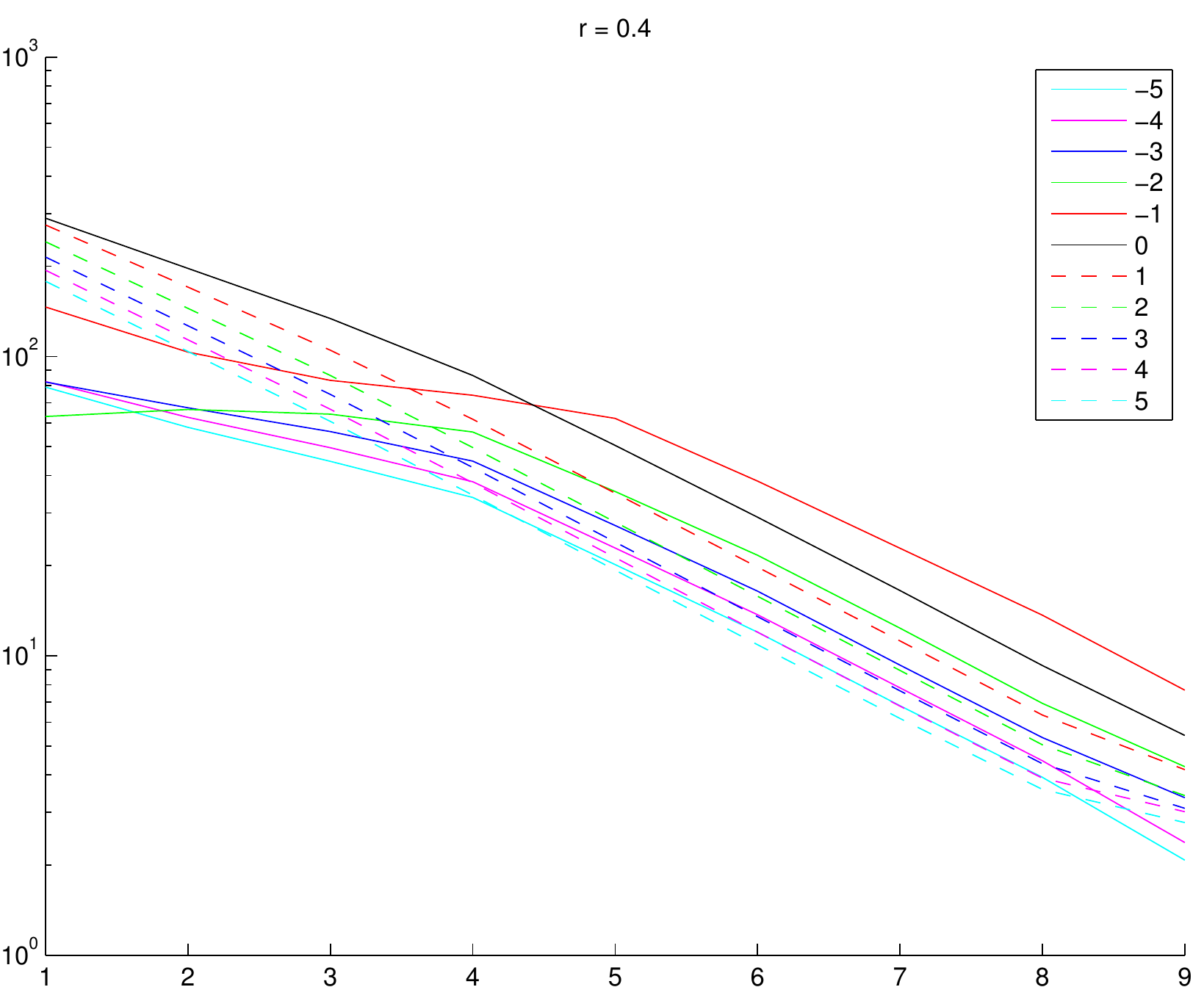}
	\includegraphics[width=0.42\textwidth]{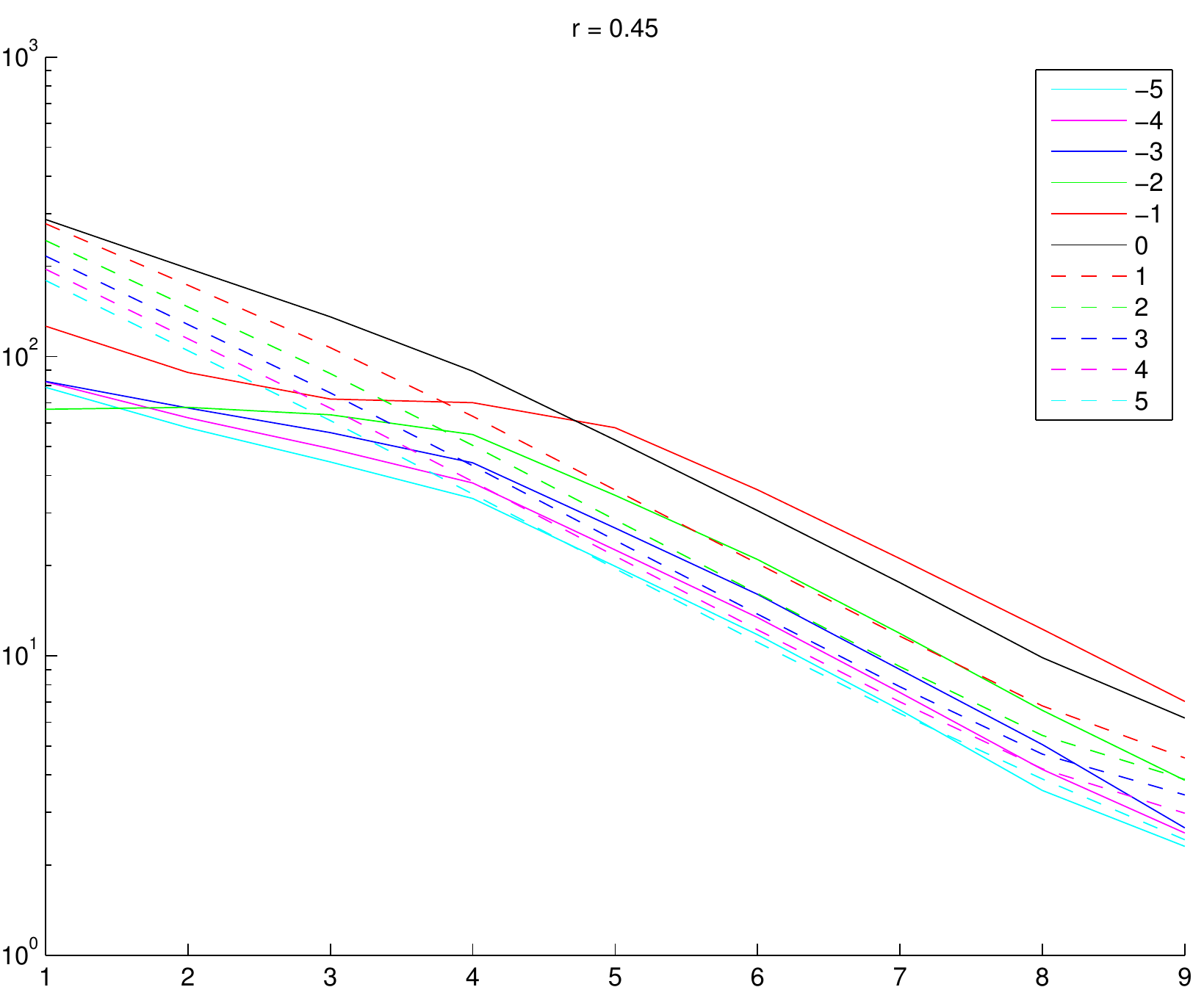}
	\end{center}
	\caption{Decay rate of bendlet transform for circles with different radii (from top left to bottom right): r = 0.10, 0.15, 0.20, 0.25, 0.30, 0.35, 0.40, 0.45 (with respect to a normalized image domain $[-1.0,1.0] \times [-1.0,1.0]$).}
	\label{fig:analyse_curvature}
\end{figure}

\cl{To numerically verify the results of the last section, we implemented bendlets as an extension to the ShearLab toolbox~\cite{Kutyniok2016b}.}

In our Shearlab implementation, we use a separable generator $\psi^1\otimes \phi^1$ with a spline for the low-frequency window $\phi^1$ and a Daubechies wavelet with eight vanishing moments for the high frequency one $\psi^1$.
The bendlet elements are defined in the spatial domain and ``rasterized'' onto the pixel domain with shearing and bending being implemented there.
Transforms are either computed in the frequency domain, following the existing implementation in Shearlab, or in the spatial domain.
The latter one is used for the computation of transforms up to fine levels since it allows us to exploit the compact support of the elements to keep the computations tractable. (We compute transforms up to level 9, corresponding to an image resolution of $16384 \times 16384$ pixels. Such a resolution is out of reach using the existing ShearLab implementation even on high-end computers.)

Experimental results demonstrating the decay rates of the bendlet transform for $\alpha = 0.335$ as a function of curvature are presented in Fig.~\ref{fig:analyse_curvature}.
As input signal we used circles with varying radii measured on a unit square domain.
In the figure one sees the expected curvature sensitivity of the decay rate as a function of the bending parameter, i.e. for small radius and large curvature only the coefficients for elements with correspondingly large bending decay slowly.
As the radius increases, and hence the curvature becomes smaller, bendlets with smaller bending parameter decay slowest.
Results for experiments to even higher levels than in Fig.~2, \cl{obtained with Mathematica for an analytic signal and analytic bending,} can be found in the supplementary material.


\section{Conclusion and Future Work}
\label{sec:conclusion}

\cl{The present paper introduced bendlets, a second-order shearlet \ms{system} based on scaling, translation, shearing, and bending of a compactly supported generator.
The associated bendlet transform enables to classify 2D singularities in an image more precisely than ordinary first-order shearlets, as is exemplified by
the classification result in Theorem~\ref{thm:mainClass} that also includes curvature.}

\cl{Curvature plays a vital role not only in 2D but also for the analysis and description of 3D objects.
An extension of our construction to 3D could hence be useful,
potentially following existing work by Kutyniok and Petersen~\cite{Kutyniok2016}.}

It would also be interesting to investigate properties of and applications for third or even higher-order shearlets, whose constructions we also detailed in this paper. \cl{For example, }\ms{if one is interested in higher-order classification of boundaries, \cl{as could be used in bag-of-feature type algorithms~\cite{Varma2005},} these transforms might provide a useful tool.
For these applications also no frame or approximation properties are required and they could hence provide a first application also of bendlets.}

\ms{As another line of future research, approximation properties of bendlets and their higher-order siblings should be analyzed.}
As the following heuristic argument demonstrates, discrete bendlet systems are, through their adaptivity to curved line singularities, promising candidates to attain good approximation rates and we expect \cl{discrete} higher-order shearlets to provide $N$-term approximation rates beyond the $N^{-2}$ barrier for piecewise smooth images with boundaries of regularity $C^k$, $k>2$.


\ms{
Consider the bivariate windowed signal
$h=\chi_{\{x_1\ge px^2_2\}}\omega$ with $p\in\R$ and
$\omega\in C^\infty_c(\R^2)$, featuring a quadratic discontinuity.
Via the transformation $h\mapsto h(Q_p^{-1}(\cdot))$ with $Q_p:(x_1,x_2)\mapsto(x_1-px_2^2,x_2)$, the signal becomes $\tilde{h}=\chi_{\{x_1\ge 0\}}\tilde{\omega}$, with modified window $\tilde{\omega}=\omega(Q^{-1}_p(\cdot))\in C_c^\infty(\R^2)$.
The discontinuity of $\tilde{h}$ is \cl{now} straight, for which
according to~\cite[Cor.~5.4]{MS2016} ordinary $0$-scaled representation systems, such as $0$-curvelets or $0$-shearlets,
allow quasi-optimal approximation.} 

\ms{The first-order $0$-shearlets from $\mathrm{SH}_{\psi}^{(1,0)}$ are
functions of the form $\psi_{a,s,0,t}=\pi^{(1,0)}(a,s,t)\psi$, where $a\in\R^{+}$, $l\in\R$ and $t=(t_1,t_2)\in\R^2$.
Due to the relation $\psi_{a,s,0,t}(Q_p(\cdot))=\psi_{a,s+2pt_2,p,t_1+pt_2^2,t_2}(\cdot)$,
the back transformation $\psi\mapsto\psi(Q_p(\cdot))$
turns $\mathrm{SH}_{\psi}^{(1,0)}$ into the system $\{ \psi_{a,s,p,t} \}_{(a,s,t)\in\R^{+}\times\R\times\R^2}$ of second-order $0$-shearlets with fixed bending $p$.
As a consequence, this subsystem of $\mathrm{SH}_{\psi}^{(2,0)}$ is equally well-adapted to resolve $h$ as $\mathrm{SH}_{\psi}^{(1,0)}$ is
for the resolution of $\tilde{h}$.}

%


\ms{An important question connected to the approximation properties of discrete bendlet systems is whether they form a frame.
Once the frame property has been verified, approximation rates can be established via the decay rates of the continuous transform, which were theoretically as well as numerically analyzed in the present paper. However,
the analysis of the frame property appears to be difficult.
}

\pp{
A natural approach would be first to understand under which circumstances bendlets form a continuous frame. Afterward, the theory of \cite{contFrames} ensures
that a \ms{sufficiently dense discretization leads to} a discrete frame. A standard method to establish the continuous frame property is to use the underlying group structure, understand the transform as a unitary representation, and demonstrate that this representation is square-integrable. However, in the case of the bendlet transform we do not have an underlying group. Thus it will be necessary to establish a different approach.
}

\section*{Acknowledgements}

C. Lessig was been partially supported by the ERC through grant ERC-2010-StG 259550 (``XSHAPE'').
He would also like to thank Eugene Fiume for continuing support.
P. Petersen acknowledges support from the DFG Collaborative Research Center TRR 109 ``Discretization in Geometry and Dynamics''. He would like to thank Thomas Fink for helpful discussions on the topic. Furthermore, P. Petersen and M. Sch{\"a}fer express their gratitude for the support of the Berlin Mathematical School.

\section*{References}

\bibliographystyle{alpha}
\bibliography{Bendlets}

\begin{thebibliography}{10}
\expandafter\ifx\csname url\endcsname\relax
  \def\url#1{\texttt{#1}}\fi
\expandafter\ifx\csname urlprefix\endcsname\relax\def\urlprefix{URL }\fi
\expandafter\ifx\csname href\endcsname\relax
  \def\href#1#2{#2} \def\path#1{#1}\fi

\bibitem{Nitzberg1993}
M.~Nitzberg, D.~Mumford, T.~Shiota, Vol. 662 of Lecture Notes in Computer
  Science, Springer Berlin Heidelberg, Berlin, Heidelberg, 1993.

\bibitem{Candes2005a}
E.~J. Cand{\`{e}}s, D.~L. Donoho, {Continuous curvelet transform: I. Resolution
  of the Wavefront Set}, Appl. Comput. Harmon. Anal. 19~(2) (2005) 162--197.

\bibitem{Candes2005b}
E.~J. Cand{\`{e}}s, D.~L. Donoho, {Continuous curvelet transform: II.
  Discretization and Frames}, Appl. Comput. Harmon. Anal. 19~(2) (2005)
  198--222.

\bibitem{Do2005a}
M.~Do, M.~Vetterli, {The contourlet transform: an efficient directional
  multiresolution image representation}, IEEE Trans. Image Process. 14~(12)
  (2005) 2091--2106.

\bibitem{Kutyniok2012}
G.~Kutyniok, D.~Labate (Eds.), {Shearlets: Multiscale Analysis for Multivariate
  Data}, Applied and Numerical Harmonic Analysis, Birkh{\"{a}}user, Boston,
  2012.

\bibitem{GuoL2009detect2D}
K.~Guo, D.~Labate, Characterization and analysis of edges using the continuous
  shearlet transform, SIAM J. Imaging Sci. 2~(3) (2009) 959--986.

\bibitem{Kutyniok2016}
G.~Kutyniok, P.~Petersen, Classification of edges using compactly supported
  shearlets, Appl. Comput. Harmon. Anal. 42~(2) (2017) 245--293.

\bibitem{GKKSAlpha16}
P.~Grohs, S.~Keiper, G.~Kutyniok, M.~Sch\"{a}fer, $\alpha$-molecules, Appl.
  Comput. Harmon. Anal. 42 (2016) 297--336.

\bibitem{Donoho2001}
D.~L. Donoho, Sparse components of images and optimal atomic decomposition,
  Constr. Approx. 17 (2001) 353--382.

\bibitem{CD04}
E.~J. Cand\`es, D.~L. Donoho, New tight frames of curvelets and optimal
  representations of objects with {$C^2$} singularities, Comm. Pure Appl. Math.
  56 (2004) 219--266.

\bibitem{GKKScurve2014}
P.~Grohs, S.~Keiper, G.~Kutyniok, M.~Sch\"{a}fer, Cartoon approximation with
  $\alpha$-curvelets, J. Fourier Anal. Appl. 22~(6) (2015) 1235--1293.

\bibitem{MS2016}
M.~Sch\"{a}fer, The role of $\alpha$-scaling for cartoon
  approximationArXiv:1612.01036.

\bibitem{GKL05}
K.~Guo, G.~Kutyniok, D.~Labate, Sparse multidimensional representations using
  anisotropic dilation and shear operators, in: Wavelets and Splines (Athens,
  GA, 2005), Nashboro Press, Nashville, TN, 2006, pp. 189--201.

\bibitem{Kittipoom2011}
P.~Kittipoom, G.~Kutyniok, W.-Q. Lim, {Construction of Compactly Supported
  Shearlet Frames}, Constr. Approx. 35~(1) (2011) 21--72.

\bibitem{Guo2007}
K.~Guo, D.~Labate, {Optimally Sparse Multidimensional Representation Using
  Shearlets}, SIAM J. Math. Anal. 39~(1) (2007) 298--318.

\bibitem{Kutyniok2010}
G.~Kutyniok, W.-Q. Lim, Compactly supported shearlets are optimally sparse, J.
  Approx. Theory 163~(11) (2011) 1564--1589.

\bibitem{Kutyniok2012a}
G.~Kutyniok, J.~Lemvig, W.-Q. Lim, {Optimally Sparse Approximations of 3D
  Functions by Compactly Supported Shearlet Frames}, SIAM J. Math. Anal. 44~(4)
  (2012) 2962--3017.

\bibitem{Kutyniok2008}
G.~Kutyniok, D.~Labate, {Resolution of the wavefront set using continuous
  shearlets}, Trans. Amer. Math. Soc. 361~(5) (2008) 2719--2754.

\bibitem{Grohs2011wavefront}
P.~Grohs, Continuous shearlet frames and resolution of the wavefront set,
  Monatsh. Math. 164~(4) (2011) 393--426.

\bibitem{Fell2015}
J.~Fell, H.~F{\"u}hr, F.~Voigtlaender, Resolution of the wavefront set using
  general continuous wavelet transforms, J. Fourier Anal. Appl. (2015) 1--62.

\bibitem{GuoLL2009detect2D}
K.~Guo, D.~Labate, W.-Q. Lim, Edge analysis and identification using the
  continuous shearlet transform, Appl. Comput. Harmon. Anal. 27~(1) (2009)
  24--46.

\bibitem{YiLEK2009edgesnumeric}
S.~Yi, D.~Labate, G.~R. Easley, H.~Krim, A shearlet approach to edge analysis
  and detection, IEEE Trans. Image Process. 18~(5) (2009) 929--941.

\bibitem{GuoL2011detect3D}
K.~Guo, D.~Labate, Analysis and detection of surface discontinuities using the
  3{D} continuous shearlet transform, Appl. Comput. Harmon. Anal. 30~(2) (2011)
  231--242.

\bibitem{HouL2016curvesinsurfaces}
R.~Houska, D.~Labate, Detection of boundary curves on the piecewise smooth
  boundary surface of three dimensional solids, Appl. Comput. Harmon. Anal.
  40~(1) (2016) 137--171.

\bibitem{GuoL2016piecewisesmooth}
K.~Guo, D.~Labate, Characterization and analysis of edges in piecewise smooth
  functions, Appl. Comput. Harmon. Anal. 41~(1) (2016) 139--163.

\bibitem{Kutyniok2016b}
G.~Kutyniok, W.-Q. Lim, R.~Reisenhofer, {ShearLab 3D: Faithful Digital Shearlet
  Transforms based on Compactly Supported Shearlets}, ACM Trans. Math. Softw.
  42~(1) (2016) 5:1--5:42.

\bibitem{Donoho2009}
D.~L. Donoho, A.~Maleki, I.~U. Rahman, M.~Shahram, V.~Stodden, {Reproducible
  Research in Computational Harmonic Analysis}, Computing in Science and
  Engineering 11~(1) (2009) 8--18.

\bibitem{ShearletGroup}
S.~Dahlke, G.~Kutyniok, P.~Maass, C.~Sagiv, H.-G. Stark, G.~Teschke, The
  uncertainty principle associated with the continuous shearlet transform, Int.
  J. Wavelets Multiresolut. Inf. Process. 6~(2) (2008) 157--181.

\bibitem{DKST2009}
S.~Dahlke, G.~Kutyniok, G.~Steidl, G.~Teschke, Shearlet coorbit spaces and
  associated {B}anach frames, Appl. Comput. Harmon. Anal. 27~(2) (2009)
  195--214.

\bibitem{Guo2006}
K.~Guo, G.~Kutyniok, D.~Labate, {Sparse Multidimensional Representations using
  Anisotropic Dilation and Shear Operators}, in: Wavelets and Splines, Nashboro
  Press, Athens, GA, 2006.

\bibitem{Varma2005}
M.~Varma, A.~Zisserman,
  \href{http://link.springer.com/10.1023/B:VISI.0000046589.39864.ee}{{A
  Statistical Approach to Texture Classification from Single Images}}, Int. J.
  of Comput. Vision 62~(1/2) (2005) 61--81.
\newblock \href {http://dx.doi.org/10.1023/B:VISI.0000046589.39864.ee}
  {\path{doi:10.1023/B:VISI.0000046589.39864.ee}}.
\newline\urlprefix\url{http://link.springer.com/10.1023/B:VISI.0000046589.39864.ee}

\bibitem{contFrames}
M.~Fornasier, H.~Rauhut,
  \href{http://dx.doi.org/10.1007/s00041-005-4053-6}{Continuous frames,
  function spaces, and the discretization problem}, J. Fourier Anal. Appl.
  11~(3) (2005) 245--287.
\newblock \href {http://dx.doi.org/10.1007/s00041-005-4053-6}
  {\path{doi:10.1007/s00041-005-4053-6}}.
\newline\urlprefix\url{http://dx.doi.org/10.1007/s00041-005-4053-6}

\end{thebibliography}


\end{document}